\documentclass[12pt,twoside,a4paper]{article}
\author{Piotr Bania}
\title{Information based approach to stochastic control problems}

\usepackage{lmodern}
\usepackage{amsmath}
\usepackage{indentfirst}
\usepackage{amsmath,amssymb}
\usepackage{microtype}
\usepackage{graphicx}
\usepackage{amsmath}
\DisableLigatures{encoding = *, family = * }
\usepackage{fancyhdr}
\usepackage{pstricks,graphicx}
\usepackage{amssymb}
\usepackage{float}

\usepackage{amsthm}
\theoremstyle{theorem}
\newtheorem{twr}{Theorem}
\newtheorem{lem}{Lemma}

\newtheorem{defi}{Definition}

\usepackage{color}
\usepackage{hyperref}

\usepackage{geometry}

\newgeometry{tmargin=2.5cm, bmargin=2.5cm, headheight=14.5pt, inner=3cm, outer=2.5cm} 

\definecolor{color02}{rgb}{0.00,0.00,1.00}
\usepackage{natbib}

\begin{document}
\title{Information based approach to stochastic control problems}
\date{November 06, 2019}
\author{Piotr Bania\footnote{Department of Automatic Control and Robotics, AGH University of Science and Technology, Al. A. Mickiewicza 30, 30-059 Krakow, Poland, e-mail: pba@agh.edu.pl. This is a preprint of an article accepted for publication in \href{https://www.amcs.uz.zgora.pl}{\textbf{International Journal of Applied Mathematics and Computer Science, AMCS, Vol. 30, No. 1, 2020.}}}}
\maketitle

\begin{abstract}      
An information based method for solving stochastic control problems with partial observation has been proposed. First, the information-theoretic lower bounds of the cost function has been analysed. It has been shown, under rather weak assumptions, that reduction of the expected cost with closed-loop control compared to the best open-loop strategy is upper bounded by non-decreasing function of mutual information between control variables and the state trajectory. On the basis of this result, an \textit{Information Based Control} method has been developed. The main idea of the IBC consists in replacing the original control task by a sequence of control problems that are relatively easy to solve and such that information about the state of the system is actively generated. Two examples of the operation of the IBC are given. It has been shown that the IBC is able to find the optimal solution without using dynamic programming at least in these examples. Hence the computational complexity of the IBC is substantially smaller than complexity of dynamic programming, which is the main advantage of the proposed method.
\end{abstract}


\section{Introduction}

 Optimal controller synthesis in stochastic systems with partial observation can be performed by dynamic programming (DP). Unfortunately, despite the theory of DP is well developed (\cite{Zabczyk:1996}), its computational complexity grows exponentially with the number of variables and time steps. As a consequence the problem is practically intractable.

To overcome the curse of dimensionality, a number of approximate methods has been developed. \textit{Separation principle} and \textit{certainty equivalence} assumption has been used by \cite{FilUn:2004}, \cite{AstWitt:1995},  \cite{Tse:1974}, \cite{BshTse:1976}. As part of the theory of Partially-Observable Markov Decision Processes (POMDP), various \textit{policy-iteration} or \textit{value-iteration} methods were developed by  \cite{Thrun:2000}, \cite{Porta:2006}, \cite{Brechtel:2013}, \cite{Dolgov:2017}, \cite{Zhao:2019} and many other researchers. These methods were initially developed for systems with finite number of states and then adopted to more general problems with smooth dynamics. Therefore, as the number of variables and time steps increases, they suffer from the curse of dimensionality. Thus, there is still a need to develop methods with less computational complexity.

Analysis of known optimal solutions (\cite{Zabczyk:1996}, \cite{FilUn:2004}, \cite{AstWitt:1995}, \cite{Tse:1974}, \cite{BshTse:1976}, \cite{Bania:2017}), suggests that active exchange of information between controller and the system is distinctive feature of the optimal controllers (cf. \cite{Bania:2018}).
Relationships between control of dynamical systems and available information are fundamental for understanding of stochastic control theory. Since the pioneering work of \cite{Feldbaum:1965} the connections between control and information theory are intensively studied. \cite{Hijab:1984} showed that concept of entropy appears naturally in dual control. The entropic formulation of stochastic control has been given by \cite{Saridis:1988} and \cite{Tsai:1992}. The works of \cite{Banek:2010} and \cite{BanKoz:2011} suggests that information exchange, entropy reduction and stochastic optimality are related to each other. An information and entropy flow in control systems has been analysed in the papers of \cite{MitNew:2005} and \cite{SagUed:2013}. The controllability, observability and stability of linear control systems with limitations of information contained in the measurements were investigated by \cite{TatMit:2004}. \cite{TL:2004} showed that controllability and observability can be defined using the concepts of information theory. One of the most relevant results related to the subject of this article is the inequality of \cite{TL:2000}. They proved that the one-step reduction in entropy of the final state is upper bounded by the mutual information between the control variables and the current state of the system. \cite{DS:2013} suggested how to extend this result to more general cost functions. 

The main contribution of this paper is as follows. First, the open and closed-loop strategy is defined in terms of mutual information between the system trajectory and control variables.  Next, it has been proved under relatively weak assumptions, that 
$$J_{open}-J_{closed}(\varphi)\leq \rho(I(X;U|\varphi)),\eqno (0)$$
where,   $J_{open}$ is the expectation of the cost corresponding to the best open-loop control, $J_{closed}$ is the expectation of the cost corresponding to any closed-loop strategy $\varphi$ and $I(X;U|\varphi)$ is the mutual information between the system trajectory and control variables under the strategy $\varphi$. Function $\rho$ is non-decreasing and $\rho(0)=0$. Additionally, we prove, that under slightly stronger assumptions, $\rho$ is bounded by linear function.  Hence the condition $I(X;U)>0$, is necessary for reduction of the cost below the best open-loop cost. On the basis of inequality (0), an \textit{Information Based Control} (IBC) has been proposed for finding an approximate solution of stochastic control problems. The phrase "approximate solution" means that the proposed method is able to find strategy no worse than \textit{open-loop feedback optimal} (OLFO) algorithm given by \cite{Tse:1974}.    
The main idea consists in replacing the original control task by a sequence of control problems that are relatively easy to solve and such that condition $I(X;U)>0$, can be fulfilled. This can be done by introducing a penalty function for information deficiency. As a penalty function, the predicted mutual information between the system trajectory and the measurements has been used. Similar idea has been proposed by \cite{Alpcan:2013}, however, in this article, the process noise (input disturbances) has been completely ignored, which is very strong and often unrealistic assumption. 
Additional contributions include sufficient conditions for existing the bounds of type (0), one step information-theoretic bound for quadratic cost and two examples of the operation of IBC. In both examples, the optimal solution has been found analytically by DP and then compared with the IBC solution. It has been shown, that IBC is able to find an optimal solution without using DP, which is the main advantage of the proposed method.    

The rest of the paper is organised as follows. Section 2 formulates the stochastic control problem. An information-theoretic lower bounds of the cost are given in section 3. Section 4 presents IBC and section 5 contains an examples of its application. Monte Carlo approximation of the cost function and some computational issues are discussed in section 6. Paper ends with conclusions and list of references.

\textbf{Notation.} Abbreviation $\xi \sim p_{\xi}$ means that variable $\xi$ has a density $p_{\xi}(\xi)$. Symbol $\xi\sim N(m,S)$  means that $\xi $ has normal distribution with mean \textit{m} and covariance \textit{S}. If $S>0$ then the density of normally distributed variable is denoted by $N(x,m,S)=(2\pi) ^{ -\tfrac{n}{2}} |S|^{ -\tfrac{1}{2}} \exp (-0.5(x-m)^{T} S^{ -1}(x-m)) $. Symbol $\mathrm{col}(a_{1} ,a_{2} ,...,a_{n} )$,  denote the column vector. Trace of matrix \textit{A} is denoted by $\mathrm{tr}(A) $. The inner product of matrices $A$ and $B$ is defined as $\langle A, B \rangle=\mathrm{tr}(A^TB)$. Let $ \xi\in R^{n} $ and let \textit{Q} be square matrix of size \textit{n}. Quadratic form $\xi^{T}Q\xi $ is denoted by $ |\xi |_{Q}^{2} $. 
Entropy of the variable $\xi$ is denoted by $H(\xi)$. Control strategy is denoted by $\varphi$. Symbol $H(\xi|\varphi)$ means that entropy of the variable $\xi$ is calculated with a fixed strategy $\varphi$. 

\section{Stochastic control task} Let us consider following stochastic system
\begin{align}
& x_{k+1} =f(x_k,u_k,w_k),k=0,1,...,\\
& y_{k} =h(x_k,v_k),\\
& u_{k} \in U_{ad}, U_{ad}=\lbrace {u\in R^{r} :u_{\min } \leq u\,\leq u_{\max } } \rbrace, 
\end{align}
where $x_{k} \in R^{n}$, $y_{k} \in R^{m}$, $w_{k} \in R^{n_{w}}$, $v_{k} \in R^{n_v}$, $w_{k} \sim p_w$, $v_{k} \sim p_v$. The inequalities in (3) are elementwise. It's also possible in some justified cases that $U_{ad}=R^r$.
Functions $f,h$ are $C^2$ wrt. all their arguments. The initial distribution of $x_0$ is denoted by $p_0^-(x_0)$. Variables $x_0, w_{0}, w_{1},... , w_{k}, v_{0}, v_{1},..., v_{k}$ are mutually independent for all $ k $. Measurements until time $k $  are denoted by $Y_{k} =\mathrm{col}(y_{0},y_{1} ,...,y_{k} )\in R^{m(k+1)}$. Similarly $X_{k} =\mathrm{col}(x_{0} ,x_{1} ,...,x_{k} )\in R^{ n(k+1)} $, $U_{k} =\mathrm{col}(u_{0},u_{1} ,...,u_{k} )\in R^{r(k+1)}$. The control horizon is denoted by $N \geq 1 $. We also introduce the following abbreviations $Y=Y_{N-1}, U=U_{N-1}, X=X_{N-1}$.

Let $B(R^{Nm},R^{Nr})$ be the set of all bounded maps form $R^{Nm}$ into $R^{Nr}$. If $f_1, f_2\in B$, $\alpha, \beta \in R$ then $\alpha f_1+\beta f_2 \in B$. Hence $B$ is linear space. The set $B$ with the norm 
$\Vert f \Vert _{\mathbf{B}}= \sup_{Y\in R^{Nm}} \Vert f(Y)\Vert _{R^{Nr}}$,
is a Banach space, which will be denoted by $\mathbf{B}$ and call the space of strategy. The measurable map
\begin{equation}
\varphi_{k} :R^{m(k+1)} \rightarrow U_{ad}, u_{k} =\varphi _{k} (Y_{k}),
\end{equation}
is control strategy at time $k$. Let $U_{ad}^N=(U_{ad}\times U_{ad}\times ,....,\times U_{ad})_{N \mathrm{times}}$. The map 
\begin{equation}
\varphi:R^{mN} \rightarrow U_{ad}^N\subset R^{Nr}, U =\varphi(Y),
\end{equation}
where
\begin{equation}
\varphi(Y)=\mathrm{col}(\varphi_0(Y_0), ...,\varphi_{N-1}(Y_{N-1})).
\end{equation} 
is admissible control strategy. The set of all admissible strategies is denoted by $S_{ad}$. It follows from (4-6) that $S_{ad}$ is bounded, closed and convex subset of $\mathbf{B}$.
Let $L:R^n\rightarrow R$ be measurable $C^2$ function and let $J: S_{ad} \rightarrow R$ denote the cost functional. We are looking for a strategy $\varphi \in S_{ad}$, that minimizes the functional
\begin{equation}
J(\varphi)=E\lbrace L(x_N)|\varphi\rbrace,
\end{equation} 
where the expectation is calculated wrt. $x_{0}$, $w_{0},..,w_{N-1}$, $v_{0} ,...,v_{N-1}$. The optimal strategy will be denoted by $\varphi^{\ast}$ and the abbreviation $J(\varphi^{\ast})=J^{\ast}$ will be used. We will assume that $\varphi^{\ast}$ exists. Optimal control corresponding to realization of $Y_{k}$  will be denoted by $u_{k}^{\ast } =\varphi_k^{\ast } (Y_k)$.
\section{Information-theoretic lower bounds of the cost function}

If the strategy $\varphi\in S_{ad}$ is fixed, then relations between random variables $X, Y, U$ are described by their joint density $p(X, Y, U|\varphi)$. In particular, if $p(X,U|\varphi)=p(X|\varphi)p(U|\varphi)$, then $X$ and $U$ are independent and information contained in measurements $Y$ is not utilized. This is open-loop control strategy. Reduction of the cost (7), compared to the open-loop, is possible only if $X$ and $U$ are dependent. The natural measure of dependency is mutual information. We will show below that the cost  (7) is lower-bounded by some non-increasing function of mutual information between $X$ and $U$.

\subsection{General bounds}
The mutual information between $X$ and $U$ is given by
\begin{equation}
I(\varphi)=H(X|\varphi)-H(X|U,\varphi),
\end{equation}
where the entropies $H(X|\varphi)$, $H(X|U,\varphi)$, are defined in usual way i.e.
\begin{align}
&H(X|\varphi)= E(-\mathrm{ln}p(X|\varphi)),\\
&H(X|U,\varphi)=E(-\mathrm{ln}p(X|U,\varphi)).
\end{align}
 The expected value in (10) is calculated wrt.  $X$ and $U$.
\begin{defi}{}
The strategy $\varphi$ is an open-loop strategy if, and only if, $I(\varphi)=0$. Otherwise $\varphi$ will be called closed-loop or feedback strategy.
 \label{definition1}
\end{defi}

Let $s\in R, s\geq0$. The set
\begin{equation}
\Omega(s)=\lbrace \varphi\in S_{ad}:I(\varphi)\leq s\rbrace
\end{equation} 
contains all strategies for which the information $I(\varphi)$ is not greater than $s$. Let 
$\varphi\in S_{ad}$ be the constant map. Since $\varphi$ is constant then $U$ and $Y$ are independent and $I(\varphi)=0$. Hence $\Omega(s)$ is non-empty for all $s \geq 0$. Consider now a family of optimization problems
\begin{equation}
\inf_{\varphi\in\Omega(s)}J(\varphi).
\end{equation}
The optimal solution of (12) will be denoted by $\varphi_s^\ast$ and it is assumed that $\varphi_s^\ast$ exist for all $s$. The minimum open-loop cost is defined as
\begin{equation}
J_o=\inf_{\varphi\in\Omega(0)}J(\varphi).
\end{equation}     
\begin{lem}{} If the solution of (12) exists for all $s\geqslant 0$, then 
there exist non-decreasing, bounded function 
$\rho:[0,\infty)\rightarrow [0,J_o-J^{\ast}],$
$\rho(0)=0,$  
such that inequality
\begin{equation}
J_o-J(\varphi) \leq \rho (I(\varphi))
\end{equation}
holds for all $\varphi\in S_{ad}$.
\end{lem}
\begin{proof}{} Let us define
\begin{equation}
\rho(s)=\sup_{\varphi\in\Omega(s)}\left( J_o-J(\varphi)\right).
\end{equation}
For every $t,s\geq0$ we have $\Omega(s)\subset \Omega(s+t)$ hence $\rho$ is non-decreasing. If $s=0$ then by the formula (13) we have
$$\rho(0)=\sup_{\varphi\in\Omega(0)}\left( J_o-J(\varphi)\right) =J_o-J_o=0. $$
Since $\varphi\in \Omega(I(\varphi))$, then
\begin{equation}
J_o-J(\varphi)\leq\sup_{\psi\in\Omega(I(\varphi))}\left( J_o-J(\psi)\right)=\rho(I(\varphi)),
\end{equation}
which proves (14).
\end{proof}
It follows from (14) that $J(\varphi)<J_o\Rightarrow I(\varphi)>0$, but function $\rho$ in (14) can be very irregular. To obtain more accurate bound, additional conditions are needed. Let 
\begin{equation}
d(\Omega(0), \varphi)=\inf_{\psi\in \Omega(0)}||\psi-\varphi||,
\end{equation}
denote the distance between $\Omega(0)$ and $\varphi$. 
\begin{twr}{}
If there exists numbers $L_I, L_J>0$, such that
\begin{align}
&|J(\varphi)-J(\varphi_1)|\leqslant L_J||\varphi_1-\varphi||, \varphi, \varphi_1\in S_{ad},\\
&I(\varphi)\geqslant L_Id(\Omega(0),\varphi), \varphi\in S_{ad},
\end{align}
then there exist number $q>0$, that
\begin{equation}
J_o-J(\varphi)\leqslant qI(\varphi), \varphi\in S_{ad}.
\end{equation}
\end{twr}
\begin{proof}{}
Let $\varphi\not\in\Omega(0)$ and let $\varphi_1\in\Omega(0)$ be such that $||\varphi_1-\varphi||=d(\Omega(0), \varphi)$. If $qL_I-L_J\geqslant 0$, then on the basis of (18), (19) and (13) we get
\begin{equation}
\begin{split}
&J(\varphi)+q I(\varphi)\geqslant J(\varphi_1)-L_J||\varphi_1-\varphi||+q L_Id(\Omega(0), \varphi)=\\
&J(\varphi_1)+(q L_I-L_J)||\varphi_1-\varphi||\geqslant J(\varphi_1)\geqslant J_o, \varphi\not\in\Omega(0).
\end{split}
\end{equation}
If $\varphi\in\Omega(0)$, then $I(\varphi)=0$ and it follows from (13) that $J(\varphi)\geqslant J_o$. Hence (20) holds for all $\varphi\in S_{ad}$.
\end{proof}
\textbf{Remark.} Data processing inequality (\cite{Cover:2006}, p.34) says that $I(X;F(Y))\leq I(X;Y)$, for any function $F$. Since  $U=\varphi(Y)$ then
\begin{equation}
I(\varphi)=I(X;U|\varphi)\leqslant I(X;Y|\varphi).
\end{equation}
As a consequence, lemma 1 and theorem 1 will still be true if we use $I(X;Y|\varphi)$, instead of $I(\varphi)$.

Since $S_{ad}$ is bounded and closed then Lipschitz continuity assumption (18) is not very restrictive. The assumption (19) says that
information must grow linearly with the distance from the set $\Omega(0)$, which seems quite natural and not very restrictive. Let us also note, that $I(\varphi)$ need not to be continuous.
\subsection{The entropy reduction of the final state}
Let us assume that the cost functional has the form 
\begin{equation}
J(\varphi)=E\lbrace-\mathrm{ln}p(x_N|\varphi)\rbrace.
\end{equation}
We will call $J(\varphi)$ the closed-loop entropy and we will write  $H(\varphi)=J(\varphi)$. The minimum open-loop entropy of the final state is denoted by $H_{o}= J(\varphi_0^{\ast})$. Touchette \& Lloyd \cite{TL:2000}, \cite{TL:2004}, showed that one-step (i.e. $N=1$) entropy reduction compared to the best open-loop strategy is upper bounded by ${I}(x_0;u_0|\varphi)$. Their inequality (in our notation), has the form
\begin{equation}
H_o-H(\varphi)\leqslant{I}(\varphi), \varphi\in{S_{ad}}.
\end{equation} 
It is fundamental limitation in control systems, but unfortunately, the multi-step ($N>1$) version of (24) is very weak (cf. \cite{Touch:2000}, p. 47, equation (3.74)). It only says, that there \textbf{exist} strategy $\varphi$, that  
\begin{equation}
H_o-H(\varphi)\leqslant\sum\limits_{k=0}^{N-1}{I}(x_k;u_k|\varphi).
\end{equation}
Since correlations between previous measurements and current control are omitted in (25), it may not be fulfilled for some $\varphi$. However, it is still possible on the basis of (24), to construct some one-step bound for (7). 
\begin{twr}{}
Let $J(\varphi)=E\lbrace L(x_1)|\varphi\rbrace$, $x_1\in{R^n}$. If $L(x_1)\geqslant{c}|x_1|^2$, $c>0$, then
\begin{equation}
J(\varphi)\geqslant {cn}(2\pi{e})^{-1}e^{2n^{-1}(H_o-I(\varphi))}.
\end{equation}
\end{twr}
\begin{proof}{}
Let $S=\mathrm{cov}(x_1,x_1|\varphi)$. Matrix $S$ fulfils the inequality $\mathrm{tr}(S)\geqslant n|S|^{\tfrac{1}{n}}$, (\cite{Cover:2006}, thm. 17.9.4, p. 680). Since Gaussian distribution maximizes entropy over all distributions with the same covariance, then it can be proved that $|S|\geqslant (2\pi e)^{-n}e^{2H(\varphi)}$, (\cite{Cover:2006}, thm. 8.6.5, p. 254). On the basis of these two inequalities and by using (24) one can obtain
\begin{equation*}
J(\varphi)\geqslant cE|x_1|^2\geqslant c\mathrm{tr}(S)\geqslant cn|S|^{\tfrac{1}{n}}
\geqslant cn(2\pi e)^{-1}e^{2n^{-1}H(\varphi)}\geqslant cn(2\pi e)^{-1}e^{2n^{-1}(H_o-I(\varphi))}.
\end{equation*}
\end{proof}
\subsection{Elementary example}
To illustrate the problem, let us consider one-dimensional system
\begin{equation}
x_1=x+u, y=x+v.
\end{equation}
Variables $x$ and $v$ are Gaussian i.e. $x\sim N(0,s_x), s_x>0$, $v\sim N(0,s_v), s_v>0$.
The cost functional has the form
\begin{equation}
J(\varphi)=E\lbrace x_1^2|\varphi\rbrace.
\end{equation}
The best open-loop strategy is $\varphi_{0}^{\ast}=0$ and $J_o=s_x$. The optimal strategy is given by linear function of $y$
\begin{equation}
\varphi^{\ast}(y)=-\frac{s_x}{s_x+s_v}y
\end{equation}
and the minimum cost is equal to 
\begin{equation}
J(\varphi^{\ast})=\frac{s_xs_v}{s_x+s_v}<s_x=J_o.
\end{equation}
Since $x_1$ is Gaussian then its open-loop entropy is given by $H_o=\tfrac{1}{2}\ln (2\pi eJ_o)$ and the inequality (26) yields
\begin{equation}
J(\varphi)\geqslant J_oe^{-2I(\varphi)},
\end{equation}
for all $\varphi$. One can check by direct calculation that
\begin{equation}
I(\varphi^{\ast})=\tfrac{1}{2}\mathrm{ln}\left( 1+\frac{s_x}{s_v} \right)
\end{equation} 
and then $J(\varphi^{\ast})=J_oe^{-2I(\varphi^{\ast})}$. Hence the bound (31) is tight. The entropy of $x_1$, under optimal strategy, is given by $H(\varphi^{\ast})=\tfrac{1}{2}\ln (2\pi eJ(\varphi^{\ast}))$ and one can check that $H_o-H(\varphi^{\ast})=I(\varphi^{\ast})$. Hence, 
the strategy (29) is also optimal for entropy reduction
\section{Information based control}
Minimum of $J(\varphi)$ can be found by dynamic programming (DP), but computational complexity of DP grows exponentially with the number of time steps and control variables. As a consequence, DP is often impractical and there is a need to construct approximate methods with lower computational complexity (cf. \cite{FilUn:2004}, pp. 14-32, \cite{AstWitt:1995}, pp. 354-370.). It is possible, on the basis of the previous section, to construct such an approximate method. The easiest way to simplify the problem is to replace the original control task with the sequence of open-loop control problems. These control problems consists in minimization of 
\begin{equation}
J_{k} (u^{(k)})=E\lbrace L(x_N)\vert Y_k, u^{(k)}\rbrace,
\end{equation}
where $u^{(k)}=\mathrm{col}(u_{k} ,...,u_{N-1})$, denote the future control sequence. Minimizer of (33) will be denoted by $\bar{u}^{(k)}(Y_k)$. To control the system, only the first element of $\bar{u}^{(k)}$ is used and the procedure is repeated in subsequent steps. Hence, the control strategy generated by sequential minimization of (33) has the form
\begin{equation}
\varphi_k(Y_k)=\bar{u}^{(k)}_1(Y_k)
\end{equation}
and this may or may not be a feedback in the sense of definition 1. The above simplification is known as \textit{Open Loop Feedback Optimal} (OLFO) and it is well known that OLFO does not generate information and can not be optimal, except linear Gaussian systems (cf. section 3.3, example 2 below, \cite{Tse:1974}, \cite{FilUn:2004}). On the other hand, it follows from section 3 and particularly from (20) and (22), that 
\begin{equation}
J(\varphi)\geqslant{J_o}-qI(X;Y|\varphi),
\end{equation}
which implies, that every controller better than open-loop, must actively generate information. Generating of information can be enforced by adding to (33), a penalty function for information deficiency. Such penalty function can be constructed by using the mutual information between future states and measurements. It is also possible to use $I(X;U)$ as a penalty, however, calculation of $I(X;U)$ is much more difficult than calculation of $I(X;Y)$. Therefore it is computationally more convenient to use $I(X;Y)$. This is basic idea of the \textit{Information Based Control} (IBC). Practically realizable implementation of the IBC is as follows. Let $X_{k}^{+} =\mathrm{col}(x_{k+1} ,...,x_{N-1})$, $Y_{k}^{+} =\mathrm{col}(y_{k+1} ,...,y_{N-1})$, denote the future states and observations. Let us define for $k=0,1,..., N-2$ 
\begin{equation}
I_{k} (u^{(k)}|Y_{k})=\int p(X_k^+,Y_{k}^{+}\vert Y_k)\mathrm{ln}\frac{p(X_k^+,Y_{k}^{+}\vert Y_k)}{p(X_k^+\vert Y_k)p(Y_{k}^{+}\vert Y_k)} dX_k^+dY_{k}^{+}.
\end{equation}
This is the mutual information between $X_k^+$  and $Y_{k}^{+}$, predicted at time $k$ and conditioned on $Y_k$. Since $y_N$ is irrelevant from the control point of view then one can assume that $I_{N-1}=0.$ Now, at every time instant we are looking for the minimum of the functional
\begin{equation}
J_{k} (u^{(k)})=E\lbrace L(x_N)\vert Y_k\rbrace -\nu_{k} I_{k} (u^{(k)}\vert Y_{k} ),
\end{equation}
where 
\begin{equation}
u_{i}^{(k)} \in{U}_{ad}, \nu_{k} \geq 0, k=0,\,...,\,N-1, N\geq 2.
\end{equation}
The expectation in (37) is calculated wrt. $x_{0}$ and $w_{k} ,...,w_{N-1}$, but not with reference to $v_{k} ,...,v_{N-1}$, which substantially simplifies the problem. Minimizer of (37) will be denoted by $\bar{u}^{(k)} $. To control the system only the first element of $\bar{u}^{(k)}$ is used and whole procedure is repeated in subsequent steps. Let us note that $\bar{u}^{(k)}$ depends on $Y_k$ as required in (4). As a consequence $X$ depends on $U$ and it's possible that IBC generates feedback strategy in the sense of definition 1. Minimizer of (37) can be considered as compromise between open-loop control (first term) and learning (second term). The intensity of learning is given by $\nu_k$. If $\nu_k=0$ then IBC becomes \textit{Open-Loop Feedback} strategy, which is generally not optimal.

\textbf{Remark.} If the system (1), (2) is linear and the disturbances are additive Gaussian white noises then mutual information in (37) does not depend on control (cf. \cite{Bania:2018}, theorem 3.1). As a consequence, application of the IBC to linear Gaussian systems with quadratic cost gives well known result i.e. Kalman filter and LQ controller.

\section{Examples}
\subsection{Example 1}
To illustrate the main idea of the IBC, let us start from the very simple example of the integrator with unknown gain. Let 
\begin{equation}
x_{k+1}=x_k+\theta u_k, y_k=x_k, \theta\in\lbrace-1,1\rbrace, x_0=1.
\end{equation}
The cost function is given by
$$ J(\varphi_0,\varphi_1)=E(x_2^2) $$
The initial distribution of $\theta$ has the form $P(\theta=-1)=p$, $P(\theta=1)=1-p$, $p\in[0,1]$. Since $\theta$ can be treated as second component of the state vector then (39) can be viewed as a special case of (1) and (2).\\
    
\textbf{The optimal solution}, obtained by dynamic programming, has the form
\begin{equation}
\varphi_0^{\ast}\neq0, \varphi_1^{\ast}(y_1)=\frac{\varphi_0^{\ast}y_1}{1-y_1}.
\end{equation}
It follows from (39) and (40) that $x_2=0.$ Hence the minimal value of the cost $J^{\ast}=0$. The observation $y_1$ contain an information about $\theta$ if, and only if, $u_0\neq 0$. Hence $I(y_1;\theta)>0$ if, and only if, $u_0\neq 0$. \\

\textbf{The information based solution.} Let $\nu_0=1$. According to (37), at the first step the following cost  
$$J(u_0,u_1)=E(x_2^2\vert y_0)-I(y_1;\theta) $$
should be minimized. Calculation of the expectation gives 
$$ J(u_0,u_1)=(u_0+u_1)^2+2(1-2p)(u_0+u_1)+1-I(y_1;\theta). $$
We know that $ I(y_1,\theta)>0 $ if, and only if, $u_0\neq0$, hence the optimal solution at the first step
$$u_0\neq0, u_1=2p-1-u_0. $$
In the second step we minimize
$$J(u_1)=E\lbrace x_2^2\vert (y_0, y_1)\rbrace=(y_1+\hat{\theta}u_1)^2$$
where
$$ \hat{\theta}=\frac{y_1-1}{u_0},$$
denote estimate of $\theta$, obtained on the basis  of $u_0$ and $y_0$.
Minimization gives
$$ u_1=\frac{u_0y_1}{1-y_1},	$$
which is just exactly the optimal solution given by (40). Thus, the IBC method allowed us to find optimal solution, without using dynamic programming. 
\subsection{Example 2}
Due to the various modelling inaccuracies, in real life applications the
parameters are not constant, but they are rather a stochastic processes. As an example of the system with
parametric noise we will first consider one-dimensional deterministic system
\begin{equation}
\dot{\eta}(t)=-a_{c}\eta (t)+(b_c+\epsilon (t))u(t)+g_{2c}\zeta (t),
\end{equation}
where $\epsilon (t)$ and $\zeta (t) $ represents changes of the gain and the input disturbances respectively. The control input is denoted by $u(t)\in R$. If we assume that $\epsilon$ is Wiener process and $\zeta$ is white noise, then (41) can be written as a system o two Ito equations
\begin{equation}
dx=(A_c(u)x+B_cu)dt+G_cdw,
\end{equation}
\begin{equation}
A_c(u)=\begin{bmatrix}
0&0\\
u&-a_{c}
\end{bmatrix}, 
B_c=\begin{bmatrix}
0\\
b_c
\end{bmatrix},
G_c=\begin{bmatrix}
g_{1c}&0\\
0&g_{2c}
\end{bmatrix}.
\end{equation}
Processes $w_1(t)$ and $w_2(t)$ are mutually independent standard Wiener processes. Parameters $a_c, b_c, g_{1c}, g_{2c}$, are positive numbers. Observation equation has the form
\begin{equation}
y_k=x_2(t_k)+v_k, k=0,1,2, ...,
\end{equation}
where
$v_k=N(0,s_v), s_v>0$, $t_k=kT_0$, $T_0>0$.  If control is piecewise constant i.e. $u(t)=u_k, t\in [t_k, t_{k+1})$, then discrete-time version of (42) and (44) is given by
\begin{align}
&x_{k+1}=A(u_k)x_k+Bu_k+\sqrt{D(u_k)}w_k,\\
&y_{k}=Cx_k+v_k,
\end{align}
where
\begin{align}
&A(u_k)=A_0+A_1u_k,\\
&D(u_k)=D_0+D_1u_k+D_2u_k^2,
\end{align} 
\begin{align}
&A_0=\begin{bmatrix}
a_1&0\\
0&a_2
\end{bmatrix},
A_1=\begin{bmatrix}
0&0\\
a_3&0
\end{bmatrix},\\
&D_0=\begin{bmatrix}
d_1&0\\
0&d_3
\end{bmatrix},
D_1=\begin{bmatrix}
0&d_2\\
d_2&0
\end{bmatrix},\\
&D_2=\begin{bmatrix}
0&0\\
0&d_4
\end{bmatrix},
B=\begin{bmatrix}
0\\
b
\end{bmatrix},
C=\begin{bmatrix}0&1
\end{bmatrix}.
\end{align}
The matrices $A, B, D$ can be calculated by using the well-known discretization rules:
\begin{equation*}
A=e^{A_c T_0}, B=\int\limits_0^{T_0} e^{A_c\tau }B_cd\tau, D=\int\limits_0^{T_0} e^{A_c\tau }G_c^2e^{A_c^T\tau }d\tau.
\end{equation*}
The input noise is a sequence of mutually independent Gaussian random variables i.e. $w_k\sim N(0,I_{2x2})$, where $I_{2x2}$ denote  identity matrix of size 2. The initial condition is given by $x_0\sim N(m_0^-,S_0^-)$. 

The cost functional is given by
\begin{equation}
J(\varphi)=\tfrac{1}{2}E\lbrace q_1x_{1,2}^2+r_0\varphi_0^2+q_2x_{2,2}^2+r_1\varphi_1^2\rbrace,
\end{equation} 
where $x_{k,2}$ denote the second component of $x_k$ and $q_k\geq 0$, $r_k>0$. Since this problem has been solved in \cite{Bania:2017}, only the main results will be presented and some laborious transformations will be omitted. To simplify the notation, we will skip some of the function's arguments, in particular instead of $m_k(Y_k,U_k), S_k(U_k), \varphi_k(Y_k)$, we will write briefly $m_k, S_k, \varphi_k$ etc. It has been shown in \cite{Bania:2018}, that joint density of $x_k, Y_k$ and the conditional density of $x_{k+1}$ are given by
\begin{align}
&p(x_{k},Y_k)=N(x_k,m_k,S_k)\prod\limits_{i=0}^kN(y_i,Cm_i^-,W_i),\\
&p(x_{k+1}|Y_k)=N(x_{k+1},m_{k+1}^-,S_{k+1}^-),
\end{align}  
where
\begin{align}
&W_i=(s_v+CS_i^-C^T),\\
&S_i=S_i^- -S_i^-C^TW_i^{-1}CS_i^-,\\
&m_i=m_i^-+S_iC^Ts_v^{-1}(y_i-Cm_i^-),\\
& m_{i+1}^-=A(u_{i})m_{i}+Bu_{i},\\
&S_{i+1}^-=A(u_{i})S_{i}A(u_{i})^T+D(u_{i}), i=0,1, ..., k.
\end{align}
Let us note, that equations (55-59) describes the Kalman filter for (45), (46).
 \subsubsection{ The optimal solution} According to (4-6), the strategy $\varphi$ consists of two mappings $u_0=\varphi_0(y_0)$ and $u_1=\varphi_1(y_0, y_1)$. The optimal solution can be found by dynamic programming. 
 It has been shown in \cite{Bania:2017}, that optimal strategy is given by
\begin{equation}
 \varphi_0^{\ast}(y_0)=\mathrm{arg}\min_{u_0\in{R}}R_0(u_0, y_0),
\end{equation}
\begin{equation}
\varphi_1^{\ast}(u_0, y_0, y_1)=-\frac{\beta_1(u_0, y_0, y_1)}{\alpha_1(u_0, y_0, y_1)},
\end{equation}
where  
\begin{align}
 &R_0(u_0, y_0)=\tfrac{1}{2}\alpha_0u_0^2+\beta_0u_0+\gamma_0+V_1(u_0, y_0),\\
&V_1(u_0, y_0)=\int N(y_1,Cm_1^-,W_1)R_1(u_0, y_0, y_1, \varphi_1^{\ast}(u_0, y_0, y_1))dy_1,\\
&R_1(u_0, y_0, y_1, \varphi_1)=\tfrac{1}{2}\alpha_1\varphi_1^2+\beta_1\varphi_1+\gamma_1,
\end{align}
\begin{align*}
&\alpha_0(y_0)=(A_1m_0+B)^TQ_1(A_1m_0+B)+\langle A_1S_0A_1^T+D_2, Q_1 \rangle +r_0,\\
&\beta_0(y_0)=(A_1m_0+B)^TQ_1A_0m_0+\tfrac{1}{2}\langle A_0S_0A_1^T+A_1S_0A_0^T+D_1, Q_1\rangle,\\
&\gamma_0(y_0)=\tfrac{1}{2}m_0^TA_0^TQ_1A_0m_0+\tfrac{1}{2}\langle A_0^TS_0A_0+D_0, Q_1\rangle,\\
&\alpha_1(u_0, y_0, y_1)=(A_1m_1+B)^TQ_2(A_1m_1+B)+\langle A_1S_1A_1^T+D_2, Q_2\rangle +r_1,\\
&\beta_1(u_0, y_0, y_1)=(A_1m_1+B)^TQ_2A_0m_1+\tfrac{1}{2}\langle A_0S_1A_1^T+A_1S_1A_0^T+D_1, Q_2\rangle,\\
&\gamma_1(u_0, y_0, y_1)=\tfrac{1}{2}m_1^TA_0^TQ_2A_0m_1+\tfrac{1}{2}\langle A_0^TS_1A_0+D_0, Q_2\rangle,\\
&Q_k=\mathrm{diag}(0, q_k), k=1, 2.
\end{align*}
Matrices $S_k$, and vectors $m_k$ are given by (56) and (57). The inner product of matrices $A$ and $B$ is denoted by $\langle A, B \rangle=\mathrm{tr}(A^TB)$.
\subsubsection{The information based solution} We will first calculate the conditional expectation. Let us denote $\xi=q_1x_{1,2}^2+r_0u_0^2+q_2x_{2,2}^2+r_1u_1^2$. After calculation of the integrals we get
\begin{equation}
E(\xi\vert Y_0)=\sum\limits_{i=1}^2\left(\mu_i^TQ_i\mu_i+r_{i-1}u_{i-1}^2+\mathrm{tr}(Q_i\Sigma_i)\right),
\end{equation}
where
\begin{align}
&\mu_{i+1}=A(u_i)\mu_i+Bu_i, \mu_0=m_0,\\&
\Sigma_{i+1}=A(u_i)\Sigma_iA(u_i)^T+D(u_i), \Sigma_0=S_0,\\&
Q_i=\mathrm{diag}(0, q_i), i=0,1. 
\end{align}
The conditional mean $m_0$ and covariance $S_0$ are given by (56), (57), where $S_0^-$, $m_0^-$ are known a priori. Now the mutual information will be calculated. It follows from (53) that
\begin{equation}
p(x_1,y_1\vert y_0)=N(x_1,m_1,S_1)N(y_1,Cm_1^-,W_1).
\end{equation}
According to section 4 we have $X_0^+=x_1$, $Y_0^+=y_1$, $Y_0=y_0$, $u^{(0)}=(u_0, u_1)^T$. Hence 
$p(X_0^+,Y_0^+\vert Y_0)=p(x_1,y_1\vert y_0)$ and calculation of the integral (36) gives 
\begin{equation}
I_0(u_0 \vert y_0)=\tfrac{1}{2}\ln \left( 1+\frac{C\Sigma_1(u_0)C^T}{s_v}\right).
\end{equation}
By the assumption we have $I_1(u^{(1)}\vert Y_1)=0$. According to (37), at the first step, we minimize the cost
\begin{equation}
J_0(u_0, u_1, y_0)=\tfrac{1}{2}\sum\limits_{i=1}^2\left(\vert\mu_i\vert_{Q_i}^2+r_{i-1}u_{i-1}^2+\mathrm{tr}(Q_i\Sigma_i)\right)
-\nu_0I_0(u_0\vert y_0).
\end{equation}
After performing calculations we get
\begin{equation}
J_0(u_0, u_1, y_0)=\tfrac{1}{2}\left(\vert\mu_1\vert_{Q_1}^2+r_0u_0^2+\mathrm{tr}(Q_1\Sigma_1(u_0))\right)-\nu_0I_0(u_0\vert y_0)+\tfrac{1}{2}\bar\alpha_0u_1^2+\bar\beta_0u_1+\bar\gamma_0,
\end{equation}
where
\begin{align*}
&\bar\alpha_0(u_0, y_0)=(A_1\mu_1+B)^TQ_2(A_1\mu_1+B)+\langle A_1S_1A_1^T+D_2, Q_2\rangle +r_1,\\
&\bar\beta_0(u_0, y_0)=(A_1\mu_1+B)^TQ_2A_0\mu_1+\tfrac{1}{2}\langle A_0S_1A_1^T+A_1S_1A_0^T+D_1, Q_2\rangle,\\
&\bar\gamma_0(u_0, y_0)=\tfrac{1}{2}\mu_1^TA_0^TQ_2A_0\mu_1+\tfrac{1}{2}\langle A_0^TS_1A_0+D_0, Q_2\rangle.
\end{align*}
The optimal value of $u_1$ is given by minimization of (72) wrt. $u_1$
\begin{equation}
u_1(u_0, y_0)=-\frac{\bar\beta_0(u_0, y_0)}{\bar\alpha_0(u_0, y_0)}.
\end{equation}
Substitution of (73) into (72) gives the analogue of equation (62)
\begin{equation}
\begin{split}
&\Psi(u_0, y_0)=J_0(u_0, u_1(u_0, y_0), y_0)=\\
&=\tfrac{1}{2}\left(\mu_1^TQ_1\mu_1+r_0u_0^2+\mathrm{tr}(Q_1\Sigma_1(u_0))\right)-
\nu_0I_0(u_0\vert y_0)+\bar\gamma_0-\frac{\bar\beta_0^2}{2\bar\alpha_0},
\end{split}
\end{equation}
where, for simplicity, the function $J_0(u_0, u_1(u_0, y_0), y_0)$ has been denoted by $\Psi(u_0, y_0)$. 	
Minimization of (74) wrt. $u_0$ gives $\bar u_0(y_0)$, which is the information-based strategy at the first step. After the first step, new information contained in $y_1$ is used by the filter (53-59) and the new state and covariance estimates ($m_1$ and $S_1$) are available. Thus, according to section 4, at the second step we minimize
\begin{equation}
J_1(u_1, Y_1)=\tfrac{1}{2}\left(\mu_2^TQ_2\mu_2+r_1u_1^2+\mathrm{tr}(Q_2\Sigma_2)\right),
\end{equation}
where
\begin{align}
&\mu_2=A(u_1)m_1+Bu_1,\\
&\Sigma_2=A(u_1)S_1A(u_1)^T+D(u_1)
\end{align}
and the control value $u_0$ (optimal or not) is treated as fixed parameter. After completing the calculations similar as above, we get
\begin{equation}
J_1(u_1)=\tfrac{1}{2}\bar\alpha_1u_1^2+\bar\beta_1u_1+\bar\gamma_1,
\end{equation}
where
\begin{align*}
&\bar\alpha_1(u_0, y_0, y_1)=(A_1m_1+B)^TQ_2(A_1m_1+B)+\langle A_1S_1A_1^T+D_2, Q_2\rangle +r_1,\\
&\bar\beta_1(u_0, y_0, y_1)=(A_1m_1+B)^TQ_2A_0m_1+\tfrac{1}{2}\langle A_0S_1A_1^T+A_1S_1A_0^T+D_1, Q_2\rangle,\\
&\bar\gamma_1(u_0, y_0, y_1)=\tfrac{1}{2}m_1^TA_0^TQ_2A_0m_1+\tfrac{1}{2}\langle A_0^TS_1A_0+D_0, Q_2\rangle.
\end{align*}
The optimal information-based solution in the second step is given by
\begin{equation}
\bar u_1=-\frac{\bar\beta_1(u_0, y_0, y_1)}{\bar\alpha_1(u_0, y_0, y_1)}.
\end{equation}
By comparing the formulas (61) and (79), we conclude that $\bar u_1$ will be equal to the optimal control $\varphi_1^{\ast}(y_0,y_1)$, provided that $u_0$ is equal to the optimal control $\varphi_0^{\ast}(y_0)$.  If this last condition is fulfilled then the optimal strategy can be recovered by the IBC. We will show below that it is possible provided that parameter $\nu_0$ in (71) is appropriately chosen. 
\subsubsection{Numerical example}
The parameters of the continuous-time system (41-43) were: $a_c=1$, $b_c=1$, $g_{1c}=g_{2c}=\sqrt{2}$, $s_v=0.01$, $T_0=0.1$. The parameters of the corresponding discrete-time system (45-51) were equal to: $a_1=1.0$, $a_2=0.90483$, $a_3=b=0.09516$, $d_1=0.2$, $d_2=9.674\:10^{-3}$, $d_3=0.18126$, $d_4=6.189\:10^{-4}$. The weights were: $r_0=r_1=10^{-3}$, $q_0=0$, $q_1=1$. The initial conditions were equal to $m_0=(0,0)^T$, $S_0=\mathrm{diag}(s_{0,1}, s_{0,2})$, $s_{0,1}=5$, $s_{0,2}=0.1$.  For simplicity, an assumption was made, that $y_0=0$. The results of numerical calculations of functions $R_0$, (62) and $\Psi$, (74), are shown in fig. 1.  
\begin{figure}[!b]
 \centering
  \includegraphics[width=0.8\textwidth]{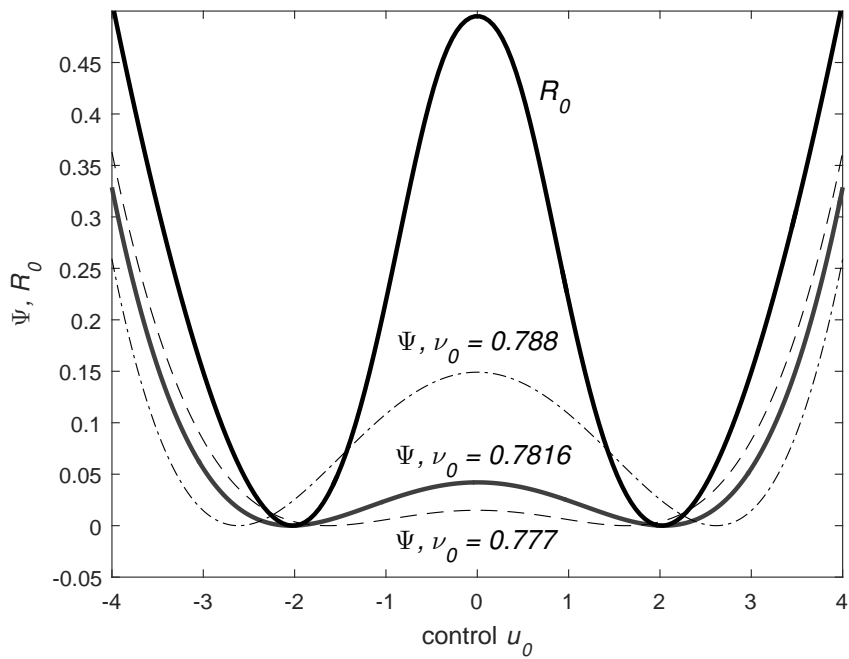}
 \caption{Graph of the functions $R _0$, (62) and $\Psi$, (74), for three values of $\nu _0$. If $ \nu_0\approx0.7816$, then function $\Psi$ has minima at the same points as $ R_0$ and the optimal strategy (60, 61) can be recovered by the IBC. For the better legibility of the picture, graphs of both functions were scaled and shifted vertically.}
  \label{fig1}
\end{figure}
The optimal control $\bar{u}_0$ is ambiguous and is equal to \ensuremath{\pm}2.0352. Although the initial 
condition is concentrated around zero, optimal control is non-zero. This is dual effect, described first by \cite{Feldbaum:1965}. Let us observe, that parameter $\nu_0$  can be chosen such that function $\Psi$, (74),  has minima at the same points as function $R_0$, (62). Hence the main conclusion that optimal feedback can be realized by the Information Based Control,  at least in this example. It's important to notice that  information based solution has been found without using dynamic programming, which substantially reduces computational complexity.
\section{Computational issues and practical implementation of the IBC}

Minimization of the cost (37) requires in advance the solution of the following problems:
\begin{enumerate}
\item
Calculation of the filtering distribution $p(x_k|Y_k)$.
\item
Calculation of the expectation in (37).
\item
Calculation of the mutual information (36).
\end{enumerate}
Filtering distribution can be calculated by Unscented Kalman Filter (UKF), Particle Filter (PF) or Gaussian Sum Filter (GSF)  (see \cite{Sarka:2013}, \cite{Alspach:1972}, for details). Since both the theory and practical implementations of these filters are well developed, we will assume below that $p(x_k|X_k)$ or its approximation is known. Let $x_{k,i}, i=1, ..., n_s$ denote samples from $p(x_k|Y_k)$ and let $x_{N,i}$, $Y_{k,i}^+$ denote the final state and observations generated by (1), (2) with initial condition $x_{k,i}$. Then it is easy to observe that samples $x_{N,i}$, $Y_{k,i}^+$ are drawn from $p(x_N|Y_k, u^{(k)})$ and $p(Y_k^+|Y_k, u^{(k)})$ respectively. Hence, the Monte Carlo approximation of the expectation in (37) is given by
\begin{equation}
E\lbrace L(x_N)\vert Y_k\rbrace\approx \frac{1}{n_s}\sum\limits_{i=1}^{n_s}L(x_{N,i}).
\end{equation}
Calculation of the mutual information (36) cannot be easily done without additional simplifications. Therefore, below we will briefly discuss special cases that are relatively easy to solve.  
Let us assume that equation (2) has the form   
\begin{equation}
y_k=h(x_k)+v_k,
\end{equation}
where $v_k\sim N(0, S_v)$. By direct calculation we get
\begin{equation}
I_{k} (u^{(k)}|Y_{k})=H_k(u^{(k)}|Y_k)-\frac{n_k}{2}\ln 2\pi e |S_v|,
\end{equation}
where $n_k$ denote size of $Y_k^+$ and
\begin{equation}
H_k(u^{(k)}|Y_k)=-\int p(Y_k^+|Y_k, u^{(k)})\ln p(Y_k^+|Y_k, u^{(k)}) dY_k^+
\end{equation}
is an entropy of $Y_k^+$, predicted at time $k$. 
Kernel Density Estimator (KDE) of $p(Y_k^+|Y_k, u^{(k)})$ has the form    
 \begin{equation}
 \hat{p}_{n_s}(Y_k^+|Y_k, u^{(k)})=\frac{1}{n_s}\sum\limits_{i=1}^{n_s}N(Y_k^+,Y_{k,i}^+,\sigma^2I_{n_k}),
 \end{equation}
 where $I_{n_k}$ is an identity matrix of size $n_k$ and the bandwidth parameter is given by 
 \begin{equation}
 \sigma=\left( \frac{4}{n_s(n_k+2)n_k^2}\right) ^{\frac{1}{n_k+4}}.
 \end{equation}
 Now, the entropy estimator can be constructed as follows 
\begin{equation}
\begin{split}
&H_k(u^{(k)}|Y_k)=E(-\ln p(Y_k^+|Y_k, u^{(k)}))\approx -\frac{1}{n_s}\sum\limits_{i=1}^{n_s}\ln \hat{p}_{n_s}(Y_{k,i}^+|Y_k, u^{(k)})=\\
&=\frac{n_k}{2}\ln (2\pi\sigma^2) -\frac{1}{n_s}\sum\limits_{i=1}^{n_s}\ln \left(\frac{1}{n_s}\sum\limits_{j=1}^{n_s}e^{-D_{i,j}} \right),
\end{split}
\end{equation}
where
\begin{equation}
D_{i,j}=\frac{1}{2\sigma^2}||Y_{k,i}^+-Y_{k,j}^+||^2.
\end{equation}
Combining  (86) and (82) we get 
\begin{equation}
I_{k} (u^{(k)}|Y_{k})\approx\frac{n_k}{2}\ln \frac{\sigma^2}{e|S_v|} -\frac{1}{n_s}\sum\limits_{i=1}^{n_s}\ln \left(\frac{1}{n_s}\sum\limits_{j=1}^{n_s}e^{-D_{i,j}} \right).
\end{equation}
On the basis of (37), (80) and (88)   
\begin{equation}
\begin{split}
&J_{k} (u^{(k)})=E\lbrace L(x_N)\vert Y_k\rbrace -\nu_{k} I_{k} (u^{(k)}\vert Y_{k} )\approx\\
&\approx \frac{1}{n_s}\sum\limits_{i=1}^{n_s}\left( L(x_{N,i})+\nu_k\ln \left(\frac{1}{n_s}\sum\limits_{j=1}^{n_s}e^{-D_{i,j}} \right)\right)-
\frac{\nu_k n_k}{2}\ln\frac{\sigma^2}{e|S_v|}.
\end{split}
\end{equation}
Since the last term in (89) does not depend on $u^{(k)}$ then finally, the cost function to be minimized is given by 
\begin{equation}
\bar{J}_{k} (u^{(k)})=\frac{1}{n_s}\sum\limits_{i=1}^{n_s}\left( L(x_{N,i})+\nu_k\ln \left(\frac{1}{n_s}\sum\limits_{j=1}^{n_s}e^{-D_{i,j}} \right)\right).
\end{equation}

Convergence conditions for (84) and (86) are given in \cite{Jiang:2017} and \cite{Joe:1989}. These conditions can be fulfilled assuming that $p_w$, $p_v$, $f$, $h$ are sufficiently regular. In particular, if $p(Y_k^+|Y_k, u^{(k)})$  is bounded, globally Lipschitz, $C^4$  and its second order partial derivatives are all upper bounded by integrable function, then (84) converges uniformly and the variance of (86) tends to zero as $n_s\rightarrow\infty$.  The convergence rate is $O(n^{-\alpha})$, $\alpha\in(0,\tfrac{1}{2}]$.

Since $f, h, L$ are $C^2$ then cost (90) is also $C^2$ wrt. $u^{(k)}$ and its gradient can be effectively calculated by solving associated adjoint equation. Then minimization of (90) can be performed by combining global search algorithms (e.g. Differential Evolution, Simulated Annealing, Genetic Algorithms) with stochastic quasi-newton methods as local solvers (\cite{Byrd:2016}).      

Control of linear system with finite number of unknown parameters and with quadratic cost function is another special case that is tractable by the IBC. An analytical formulas describing the cost function and the filtering distribution has been given in \cite{Bania:2018}. Various types of recursive filters are also analysed in \cite{BanBar:2016}, \cite{BanBara:2017}, \cite{BarBan:2017}. Computationally effective lower bound of the mutual information (36), that can be utilized to construct an upper bound of the cost, is given in \cite{Bania:2019}. Thus, in this particular case, the cost (37) and its gradient can be calculated without using Monte Carlo sampling and the control problem is relatively easy to solve.

\section{Conclusions} Lower bounds of the cost function in stochastic optimal control problems have been analysed in terms of information exchange between the system and the controller. It has been proved, under weak assumptions, that the cost function is lower bounded by some decreasing function of mutual information between the system trajectory and control variables. Under some additional regularity conditions, the lower bound obtained above is linear function of information, but the constant $q$ appearing in (20) depend on system dynamics. It also follows from theorem 1 and (22), that minimum value of the cost is determined by the capacity of the measurement channel (i.e maximal value of $I(X;Y)$). Next, on the basis of Touchette-Lloyd inequality, a new one-step lower bound (26) has been established, provided that cost function is quadratic. This bound is independent on system dynamics and in that sense universal. 

Inequalities (20) and (22) indicates that restrictions in communication between parts of the system prevent certain states from being reached. One of the examples of such phenomenon is synchronization in dynamical networks. Since the synchronization problem can be interpreted as stochastic control task, then communication constraints of the form $I(X;Y)<C$ implies that $J_o-J(\varphi)\leqslant C$. As a consequence synchronization may be lost if $C$ is too small. This was confirmed in \cite{Huang:2012}.

The conclusion resulting from the analysis of information-theoretic bounds is that feedback controller must actively (if possible) generate information about the state of the system. On the basis of these results, the \textit{Information Based Control} approach to stochastic control has been proposed. The main idea of the IBC consists in replacing of the original control problem by sequence of simpler, auxiliary control problems. The cost function to be minimized in these auxiliary problems consists in two parts: the predicted expectation of the cost conditioned on available measurements and the penalty function for information deficiency. As penalty function, the predicted mutual information between the trajectory and measurements has been used.  Hence the method enforces active generation of information about the system state and is able to generate feedback strategy.  The IBC method can be also viewed as modification of the OLFO (\cite{Tse:1974}) algorithm or as compromise between control and state estimation. 

It follows from section 6 that minimization of the cost (37) can be performed by standard optimization algorithms, without using dynamic programming. Hence the computational complexity of the IBC is substantially smaller than complexity of DP. This feature of the IBC makes the possibility of solving large-scale tasks, which is impossible with DP. It has been shown that IBC is able to find an optimal solutions, provided that learning intensity (parameter $\nu_k$) is appropriately selected. The optimal value of $\nu_k$ can be tuned experimentally but, at the current stage of research, this problem is not resolved. The ability of the IBC to find optimal solution is surprising but, due to the complexity of the problem, convergence to optimal solution is difficult to investigate and has not been proven. 

Effective calculation of the mutual information or development of its approximation is crucial issue and some methods from the optimal experimental design and fault detection theory can be adopted here (see \cite{Bania:2019}, \cite{Uci:2004}, \cite{Korbicz:2004}). It is also possible to use the information lower bound proposed by \cite{KolTra:2017}. 

Application of the IBC method to solve more realistic control problems and developing of information-based model predictive control algorithms is planned as a part of future works.
\bibliographystyle{agsm}
\bibliography{literature}
\end{document}